\documentclass[10pt, reqno]{amsart}
\numberwithin{equation}{section}
\usepackage{amssymb, color, cite}
\usepackage{hyperref}

\newcommand{\qtq}[1]{\quad\text{#1}\quad}

\newcommand{\R}{\mathbb{R}}
\newcommand{\C}{\mathbb{C}}

\newcommand{\eps}{\varepsilon}

\newcommand{\F}{\mathcal{F}}

\newtheorem{theorem}{Theorem}[section]

\newtheorem{lemma}[theorem]{Lemma}

\newtheorem{proposition}[theorem]{Proposition}

\theoremstyle{definition}
\newtheorem{definition}[theorem]{Definition}

\theoremstyle{remark}

\newcommand{\cgc}[1]{\color{red} {\tt [GC: #1]}  \color{black}}

\begin{document}

\title[Recovery of the nonlinearity]{Recovery of the nonlinearity from  \\ the modified scattering map}

\begin{abstract} We consider a class of one-dimensional nonlinear Schr\"odinger equations of the form
\[
(i\partial_t+\Delta)u = [1+a]|u|^2 u.
\]
For suitable localized functions $a$, such equations admit a small-data modified scattering theory, which incorporates the standard logarithmic phase correction.  In this work, we prove that the small-data modified scattering behavior uniquely determines the inhomogeneity $a$. 
\end{abstract}

\author{Gong Chen}
\address{Georgia Institute of Technology}
\email{gc@math.gatech.edu}

\author{Jason Murphy}
\address{Missouri University of Science \& Technology}
\email{jason.murphy@mst.edu}

\maketitle

%%%%%%%%%%%%%%%%%%%%%%
\section{Introduction}

We consider one-dimensional nonlinear Schr\"odinger equations of the form
\begin{equation}\label{nls}
\begin{cases}
(i\partial_t + \Delta) u = [1+a]|u|^2 u, \\
u|_{t=0}=u_0,
\end{cases}
\end{equation}
where the inhomogeneity $a:\R\to\R$ is a localized function of $x\in\R$.  For suitable functions $a$, equation \eqref{nls} admits a small-data modified scattering theory for initial data chosen from a weighted Sobolev space. In this paper, we prove that the modified scattering map uniquely determines the inhomogeneity $a$. 

We first describe the class of inhomogeneities considered in this work:
\begin{definition}[Admissible]\label{D:admissible}
We say $a:\R\to\R$ is \emph{admissible} if $a\in L^1\cap L^\infty$, $xa\in L^2$, and $\partial_x a\in L^1$.
\end{definition}

For admissible inhomogeneities $a$, we may obtain the following modified scattering result for small initial data in a weighted Sobolev space, which incorporates the typical logarithmic-type phase correction.  In the notation below, $\F$ denotes the Fourier transform and $e^{it\Delta}=\F^{-1} e^{-it\xi^2}\F$ is the Schr\"odinger group.

\begin{theorem}[Modified scattering]\label{T1} Let $a:\R\to\R$ be admissible in the sense of Definition~\ref{D:admissible}.  If $\|u_0\|_{H^{1,1}}$ is sufficiently small, then there exists a unique forward-global solution $u$ to \eqref{nls} and $w_+\in L_\xi^\infty$ such that
\begin{equation}\label{modified-scattering}
\lim_{t\to\infty}\biggl\| \exp\biggl\{i\int_0^t |\F e^{-is\Delta}u(s)|^2\tfrac{ds}{2s+1}\biggr\}\F e^{-it\Delta} u(t) - w_+\biggr\|_{L_\xi^\infty} = 0.
\end{equation}
\end{theorem}

Using Theorem~\ref{T1}, we may define the \emph{modified scattering map}. 

\begin{definition}[Modified scattering map]\label{D:map} Let $a$ be admissible in the sense of Definition~\ref{D:admissible}.  Given $\eps>0$, define
\[
B_\eps=\{u_0\in H^{1,1}:\|u_0\|_{H^{1,1}}<\eps\}.
\]
For $\eps$ sufficiently small, we may use Theorem~\ref{T1} to define the \emph{modified scattering map} $S_a:B_\eps\to L^\infty$ by $S_a(u_0)=w_+$, where $w_+$ is as in \eqref{modified-scattering}.
\end{definition}

Our main result shows that the modified scattering map uniquely determines the inhomogeneity $a$.

\begin{theorem}[The modified scattering map determines the nonlinearity]\label{T2} Suppose $a$ and $b$ admissible in the sense of Definition~\ref{D:admissible}. Let $S_a:B\to L^\infty$ and $S_b:B'\to L^\infty$ denote the corresponding modified scattering maps.  

If $S_a=S_b$ on $B\cap B'$, then $a\equiv b$. \end{theorem}

Theorem~\ref{T2} fits in the context of a wide body of work on the recovery of nonlinearities (and external potentials) for nonlinear dispersive equations, particularly the question of recovery from scattering data; we refer the reader to \cite{SBUW, SBS2, CarlesGallagher, EW, HMG, KMV, MorStr, Murphy, PauStr, Sasaki, Sasaki2, SasakiWatanabe, Watanabe0, Watanabe, Weder0, Weder1, Weder6, Weder3, Weder4, Weder5} for a broad selection of works in this direction.  The chief novelty in our work stems from the fact that we consider a class of equations for which the usual (unmodified) scattering \emph{fails}.  That is, the long-time behavior of solutions is not simply given by the underlying linear dynamics; instead, due to insufficient time decay in the nonlinear term, one must incorporate a logarithmic phase correction in order to describe the long-time asymptotic behavior.  Consequently, the structure of the modified scattering map is more complicated to describe.  Nonetheless, as we will explain below, this modified map suffices to uniquely determine the inhomogeneity present in the nonlinearity. 

Before discussing the proof of Theorem~\ref{T2}, let us briefly describe the proof of modified scattering for \eqref{nls} (Theorem~\ref{T1}).  Modified scattering for cubic nonlinear Schr\"odinger equations in one dimension is an important topic that has been addressed in many different settings (see e.g. \cite{HayashiNaumkin, LindbladSoffer, KaPu, IfrimTataru, MurphySurvey, ChenPusateri, ChenPusateri0, Naumkin, Naumkin2, DeiftZhou}). In the setting of \eqref{nls}, the inhomogeneous cubic term may be viewed as a short-range perturbation to the long-range nonlinearity $|u|^2 u$; indeed, the inhomogeneity $a(x)$ does not appear in the phase correction itself (cf. \eqref{modified-scattering}).  Our proof of modified scattering follows the basic scheme set out in \cite{KaPu} (based on taking the Fourier transform of the Duhamel formula and using an integrating factor to remove the non-integrable cubic part), using local smoothing estimates (similar to those appearing in \cite{ChenPusateri}) to handle the inhomogeneous cubic term.  For the details, see Section~\ref{S:direct}.

In Section~\ref{S:inverse}, we prove the main result, Theorem~\ref{T2}. Before discussing specific details of the proof, let us first recall the standard approach to recovering the nonlinearity from the usual scattering map (going back at least as far as \cite{MorStr, Strauss}).  To fix ideas, let us consider the problem of recovering an unknown, localized coefficient in a $1d$ nonlinear Schr\"odinger equation of the form
\begin{equation}\label{nls-a}
(i\partial_t + \Delta)u = a|u|^2 u, \quad u|_{t=0}=u_0.
\end{equation}
For $a\in L^1\cap L^\infty$, one can prove that the usual (unmodified) scattering behavior holds for small initial data in $L^2$ (see e.g. \cite{Murphy}); that is, there exists a map $S_a$ such that
\[
\lim_{t\to\infty} \|u(t) - e^{it\Delta} S_a(u_0)\|_{L^2} = 0,
\]
where $u$ is the solution to \eqref{nls-a}.  In fact, using the Duhamel formula, one obtains the following implicit formula for $S_a$:
\[
S_a(u_0) = u_0 -i\int_0^\infty e^{-it\Delta}[a|u(t)|^2 u(t)]\,dt. 
\]
Specializing to $u_0=\eps\varphi$ (with $\varphi\in\mathcal{S}$ and $0<\eps\ll 1$), pairing this identity with $\varphi$, and approximating $u(t)$ by $e^{it\Delta}u_0$ (the Born approximation), one can show that 
\begin{align*}
\langle S_a(\eps\varphi),\varphi\rangle & = \eps\langle \varphi,\varphi\rangle - i\eps^3\int_0^\infty\int_\R a(x)|e^{it\Delta}\varphi(x)|^4\,dx\,dt +\mathcal{O}(\eps^4).
\end{align*}
It follows that knowledge of $S_a$ suffices to determine the functionals
\begin{equation}\label{functional}
\int_0^\infty \int_\R a(x)|e^{it\Delta}\varphi(x)|^4\,dx\,dt \qtq{for}\varphi\in\mathcal{S}. 
\end{equation}
The problem then reduces to showing that knowledge of the functionals \eqref{functional} uniquely determines the coefficient $a$. 

In the setting of Theorem~\ref{T2}, the 
overall structure of the argument is similar; however, the analysis becomes more complicated due to the fact that the form of the modified scattering map is different than that of the standard scattering map.  In particular, the modified scattering map is no longer easily viewed as a perturbation of the identity. Instead, in Proposition~\ref{P:Sa}, we show that for $\varphi\in\mathcal{S}$ and $0<\eps\ll 1$, we have the expansion
\begin{align*}
\langle S_a(\eps\varphi),\hat\varphi\rangle& = \eps\langle\hat\varphi,\hat\varphi\rangle + \tfrac{1}{2i}\log(1+\tfrac{1}{2\eps})\langle |S_a(\eps\varphi)|^2S_a(\eps\varphi),\hat\varphi\rangle + \eps^3 \mathcal{Q}_\eps[\varphi] \\
&\quad - i\eps^3\int_0^\infty \int_\R a(x)|e^{it\Delta}\varphi(x)|^4\,dx\,dt + \mathcal{O}(\eps^4), 
\end{align*}
where $\hat\varphi$ is the Fourier transform of $\varphi$ and $\mathcal{Q}_\eps$ is a multilinear expression in $\varphi$ (which, importantly, is independent of $a$).  Thus, despite the more complicated structure of $S_a$, we find that knowledge of $S_a$ still essentially determines the functionals appearing in \eqref{functional}, and the problem once again reduces to showing that the functionals \eqref{functional} determine the coefficient $a$.

In earlier works (e.g. \cite{Strauss, MorStr}), this final step is completed by evaluating the functional along a sequence of test functions concentrating at a point and utilizing the dominated convergence theorem in order to determine $a$ pointwise. In the present setting, the low-power nonlinearity poses an additional challenge; indeed, we cannot use dominated convergence directly, as we cannot guarantee that $e^{it\Delta}\varphi \in L_{t,x}^4(\R\times\R)$ even for $\varphi\in\mathcal{S}$.  Instead, inspired in part by \cite{KMV}, we proceed by specializing to the case of Gaussian data, for which the free evolution may be computed explicitly.  In this way, we find that knowledge of \eqref{functional} suffices to determine the convolution $a\ast K$ for an explicit kernel $K$, and the problem reduces to verifying directly that $\hat K\neq 0$ almost everywhere.  This final step is completed by evaluating a Gaussian integral. 

The rest of this paper is organized as follows:  In section~\ref{S:Notation}, we set up notation and collect some preliminary lemmas.  In Section~\ref{S:direct}, we establish modified scattering for \eqref{nls} (Theorem~\ref{T1}).  Finally, in Section~\ref{S:inverse}, we prove the main result, Theorem~\ref{T2}. 

\subsection*{Acknowledgements} J.M. was supported by NSF grant DMS-2137217 and a Simons Collaboration Grant.  G.C. would like to thank the Department of Mathematics and Statistics at Missouri S\&T, where part of this work was completed, for its hospitality.
%%%%%%%%%%%%%%%%%%%%%
\section{Notation and preliminary results}\label{S:Notation}
We write $A\lesssim B$ to denote $A\leq CB$ for some $C>0$.  We indicate dependence on parameters via subscripts, e.g. $A\lesssim_a B$ means $A\leq CB$ for some $C=C(a)>0$. 

We write $H^{k,\ell}$ to denote the weighted Sobolev space with norm
\[
\|u\|_{H^{k,\ell}} = \|\langle\partial_x\rangle^k \langle x\rangle^\ell u\|_{L^2}, 
\]
where $\langle\cdot\rangle$ is the Japanese bracket notation, i.e. $\langle x\rangle = \sqrt{1+x^2}$.  We write $\mathcal{S}$ for Schwartz space.

We denote the Fourier transform of a function $f:\R^d\to\C$ by
\[
\F_d f(\xi) = (2\pi)^{-\frac{d}2}\int_{\R^d} e^{-ix\cdot\xi}f(x)\,dx,
\]
with the inverse Fourier transform given by
\[
\F_d^{-1} f(x) = (2\pi)^{-\frac{d}2}\int_{\R^d} e^{ix\cdot\xi}f(\xi)\,d\xi. 
\]

If $d=1$, we will omit the subscript. We also write $\F f = \hat f$ and $\F^{-1} f = \check f$.  We caution the reader that factors of $2\pi$ will be uniformly omitted throughout the computations below. 

The Schr\"odinger group is given by the Fourier multiplier operator 
\[
e^{it\Delta} = \F^{-1} e^{-it\xi^2}\F.
\]
This operator admits the factorization identity
\[
e^{it\Delta} = M(t)D(t)\F M(t),
\]
where
\[
M(t) = e^{i\frac{x^2}{4t}} \qtq{and} [D(t)f](x) = (2it)^{-\frac12}f(\tfrac{x}{2t}).
\]

The Galilean operator $J(t)$ is defined via
\begin{equation}\label{Galilean}
J(t)=x+2it\partial_x = e^{it\Delta}x e^{-it\Delta}.  
\end{equation}

Given a solution $u$ to \eqref{nls}, we will perform much of the analysis on the associated profile $f(t)=e^{-it\Delta}u(t)$.  Suitable bounds on the profile imply estimates for the solution itself, as is seen in the following lemma. 
\begin{lemma}\label{L:PW} Let $f(t)=e^{-it\Delta}u(t)$.  Then for any $0<c<\tfrac14$, 
\[
\|u(t)\|_{L_x^\infty} \lesssim_c |t|^{-\frac12}\{ \|\hat f(t)\|_{L^\infty} + |t|^{-c}\|\hat f(t)\|_{H^1}\}. 
\]
\end{lemma}

\begin{proof} We write
\begin{align*}
u(t) & = M(t)D(t)\F M(t)f(t) \\
& = M(t)D(t) \hat f(t) + M(t)D(t)\F[M(t)-1]f(t). 
\end{align*}
We now observe that
\[
\|M(t)D(t)\hat f(t)\|_{L^\infty} \lesssim |t|^{-\frac12}\|\hat f(t)\|_{L^\infty},
\]
which is acceptable. For the remaining term, we use Hausdorff--Young, the pointwise estimate
\[
|M(t)-1|\lesssim |x|^{2c}|t|^{-c},
\]
and Cauchy--Schwarz to obtain
\begin{align*}
\|M(t)D(t)\F[M(t)-1]f(t)\|_{L^\infty} & \lesssim |t|^{-\frac12-c}\| |x|^{2c} f\|_{L^1} \\
& \lesssim |t|^{-\frac12-c}\|\langle x\rangle f\|_{L^2}, 
\end{align*}
which is acceptable. 
\end{proof}

Next we introduce a smoothing estimate, which is the dual of the classical Kato smoothing estimate.  This estimate will be used to analyze the inhomogeneous cubic term. Such estimates appear in more general settings in \cite{ChenPusateri}.

\begin{lemma}\label{lem:smoothing} Let $\phi: \R \mapsto \C$ satisfy
\begin{equation}\label{phi-assumption}
|\phi(k)|\lesssim |k|^{\frac12}.
\end{equation}
Then for all $t\geq0$, we have
\begin{align}\label{eq:smoothing}
\left\Vert \int_0^t e^{-i\xi^{2}s} \phi(\xi) \hat{F}(s,\xi) \,ds\right\Vert _{L_{\xi}^{2}}\lesssim & \left\Vert F\right\Vert _{L_{x}^{1}L_{s}^{2}(\R\times[0,t])}.
\end{align}
\end{lemma}
\begin{proof}  We argue by duality.  We will first prove that
\begin{equation}\label{gc-dual}
\left\Vert  \int_\R  e^{ix\xi} \bar\phi(\xi) e^{i\xi^{2}s}h(\xi)\,d\xi\right\Vert _{L_{x}^{\infty}L_{s}^{2}(\R\times [0,t])}\lesssim \left\Vert h\right\Vert _{L^{2}}
\end{equation}
for any $h\in L^2$.  Without loss of generality, we restrict the integral to $\xi>0$. Changing variables via $\xi^{2}=\lambda$ and using Plancherel (in time) and \eqref{phi-assumption}, we obtain
\begin{align*}
\biggl\|&\int_0^\infty e^{ix\xi}\bar\phi(\xi)e^{i\xi^2 s}h(x)\,d\xi\biggr\|_{L_s^2([0,t])}^2 \\
& \lesssim \left\| \int_0^\infty e^{ix\sqrt{\lambda}}\bar\phi(\sqrt{\lambda})e^{is\lambda} h(\sqrt{\lambda})\tfrac{1}{\sqrt{\lambda}}\,d\lambda \right\|_{L^2_s(\R)}^2\\
& \lesssim \int_0^{\infty} \biggl|  \tfrac{\bar\phi(\sqrt{\lambda})}{\sqrt{\lambda}}h(\sqrt{\lambda}) \biggr|^{2}\,d\lambda \lesssim \|h\|_{L^2}^2,
\end{align*}
uniformly in $x$, which yields \eqref{gc-dual}. 

Now, given $h\in L^2$ and $F \in L_x^1 L_s^2$, we use \eqref{gc-dual} and H\"older to estimate
\begin{align*}
\biggl| & \int \int_0^t   h(\xi) e^{i\xi^{2}s} \bar\phi(\xi)  \bar{\hat{F}}(s,\xi)\,ds\,d\xi \biggr| \\
& = \left| \int_0^t \int_\R \Big( \int_\R e^{ix\xi}e^{i\xi^{2}s}\bar\phi(\xi) h(\xi)\,d\xi \Big) \bar{F}(s,x) \,dx\,ds\right| \\
& \lesssim \left\Vert \int_\R e^{ix\xi} \bar\phi(\xi)e^{i\xi^{2}s}h(\xi)\,d\xi\right\Vert _{L_{x}^{\infty} L_s^2(\R\times[0,t])} 	\big\| F \big\|_{L_{x}^{1}L_s^2(\R\times[0,t])}		\\
& \lesssim \left\Vert h\right\Vert _{L^{2}} \big\| F \big\|_{L_{x}^{1}L_s^2(\R\times[0,t])},	
\end{align*}
which implies the desired estimate.
\end{proof}
%%%%%%%%%%%%%%%%%%%%%
\section{The Direct Problem}\label{S:direct}

In this section we prove Theorem~\ref{T1}. The proof follows largely along standard lines (see e.g. \cite{KaPu}), with some modifications to handle the inhomogeneous cubic term. 

We let $u_0\in H^{1,1}$ with $\|u_0\|_{H^{1,1}}=\eps>0$, and let $u:[0,\infty)\times\R\to\C$ be the corresponding solution to \eqref{nls}. We define the profile $f(t)=e^{-it\Delta}u(t)$.  By standard well-posedness arguments and Sobolev embedding, one can derive that that
\begin{equation}\label{t01}
\sup_{t\in[0,1]}[\|u(t)\|_{H^1}+\|J(t)u(t)\|_{L^2}\bigr]\lesssim \eps. 
\end{equation}

Using \eqref{nls}, we have that
\[
i\partial_t \hat f(t,\xi) = \F e^{-it\Delta}(|u|^2 u)(\xi) + \F e^{-it\Delta}(a|u|^2 u)(\xi).
\]
In particular, we have the following straightforward estimates, which will be useful for $t\in[0,1]$: by Hausdorff--Young and Plancherel,
\begin{align*}
\|\partial_t \hat f\|_{L_\xi^\infty} \lesssim \|[1+a]|u|^2 u\|_{L_x^1} \lesssim \|u\|_{L_x^3}^3 \lesssim \|u\|_{H_x^1}^3 \lesssim \eps^3
\end{align*}
and
\begin{align*}
\|\partial_\xi \hat f\|_{L_\xi^2} & \lesssim \|J(t)([1+a]|u|^2 u)\|_{L^2} \\ 
& \lesssim \|u\|_{L^\infty}^2\|Ju\|_{L^2}+ \|u\|_{L^\infty}^3 \|t\nabla a\|_{L^2} \lesssim \eps^3.  
\end{align*}

Next, we isolate the component of $i\partial_t \hat f$ that fails to be integrable as $t\to\infty$. 
 Evaluating the Fourier transform and changing variables via $\xi-\sigma\mapsto \sigma$, we obtain
\begin{align*}
\F& e^{-it\Delta}\bigl(|u|^2 u\bigr)(\xi) \\
& = \iint e^{it[\xi^2-(\xi-\eta)^2+(\eta-\sigma)^2-\sigma^2)]}\hat f(t,\xi-\eta)\hat{\bar{f}}(t,\eta-\sigma)\hat f(t,\sigma)\,d\sigma\,d\eta \\
& = \iint e^{2it\eta\sigma}G_\xi[f(t),f(t),f(t)](\eta,\sigma)\,d\sigma\,d\eta,
\end{align*}
where
\begin{equation}\label{def:G}
G_\xi[f,g,h](\eta,\sigma):= \hat f(\xi-\eta) \hat{\bar{g}}(\eta-\xi+\sigma)\hat h(\xi-\sigma).
\end{equation}

We continue from above, using Plancherel and the identity
\[
\F_2[e^{2it\eta\sigma}] = \tfrac{1}{2t} e^{-i\frac{\eta\sigma}{2t}} 
\]
to obtain
\begin{align*}
\F e^{-it\Delta}\bigl(|u|^2 u\bigr)(\xi) & = \tfrac{1}{2t}\iint e^{-i\frac{\eta\sigma}{2s}} \F_2^{-1}\bigl\{G_\xi[f(t),f(t),f(t)]\bigr\}(\eta,\sigma)\,d\sigma\,d\eta. 
\end{align*}

Noting that $\hat{\bar{f}}(-\xi)=\bar{\hat{f}}(\xi)$, so that
\[
G_\xi[f(t),f(t),f(t)](0,0) = |\hat f(t,\xi)|^2\hat f(t,\xi),
\]
we therefore find that 
\begin{align*}
\F e^{-it\Delta}\bigl(|u|^2 u\bigr)(\xi)  & = \tfrac{1}{2t}|\hat f(t,\xi)|^2\hat f(t,\xi) \\ 
& \quad + \tfrac{1}{2t}\iint \bigl[ e^{-i\frac{\eta\sigma}{2t}}-1\bigr]\F_2^{-1}\{G_\xi[f(t),f(t),f(t)]\}(\eta,\sigma)\,d\sigma\,d\eta. 
\end{align*}

Combining the computations above, we derive that
\begin{align*}
i\partial_t f(t,\xi) & = \tfrac{1}{2t}|\hat f(t,\xi)|^2\hat f(t,\xi) \\
& \quad + \F e^{-it\Delta}\bigl(a|u|^2 u\bigr)(\xi) \\
& \quad + \tfrac{1}{2t} \iint \bigl[e^{-i\frac{\eta\sigma}{2t}}-1\bigr]\F_2^{-1}\{G_\xi[f(t),f(t),f(t)]\}(\eta,\sigma)\,d\sigma\,d\eta. 
\end{align*}

We now define
\begin{equation}\label{def:w}
w(t) = e^{iB(t)}\hat f(t),\qtq{where} B(t): = \exp\biggl\{i \int_0^t |\hat f(s)|^2\tfrac{ds}{2s+1}\biggr\}. 
\end{equation}
It follows that
\begin{align}
i\partial_t w(t,\xi) & = e^{iB(t,\xi)}\bigl\{i\partial_t f(t,\xi)-\tfrac{1}{2t+1}|\hat f(t,\xi)|^2 \hat f(t,\xi)\bigr\} \label{dtw1} \\
& = e^{iB(t,\xi)}\biggl[ \tfrac{1}{2t(2t+1)}|\hat f(t,\xi)|^2\hat f(t,\xi) \label{dtw21} \\
& \quad\quad\quad\quad + \F e^{-it\Delta}(a|u|^2 u)(\xi)\label{dtw22}\\
& \quad\quad\quad\quad +\tfrac{1}{2t}\iint\bigl[e^{-i\frac{\eta\sigma}{2t}}-1]\F_2^{-1}\{G_\xi[f(t),f(t),f(t)]\}(\eta,\sigma)\,d\sigma\,d\eta\biggr].\label{dtw23} 
\end{align}
Using \eqref{dtw1} and \eqref{t01}, we find that
\begin{equation}\label{dtwt01}
\|\partial_t w\|_{H^1} \lesssim \eps^3\qtq{uniformly for}t\in[0,1].
\end{equation}

We obtain estimates for $t\in [1,\infty)$ using a bootstrap argument. In particular, assuming that the solution satisfies estimates of the form
\begin{equation}\label{bootstrap}
\|\hat f(t)\|_{L_\xi^\infty} \leq 2C\eps\qtq{and} \|\hat f(t)\|_{H^1} \leq 2C\langle t\rangle^{\delta}\eps
\end{equation}
uniformly in $t\geq 1$, the estimates obtained below will demonstrate that the solution satisfies the improved bounds
\[
\|\hat f(t)\|_{L_\xi^\infty}\leq C\eps\qtq{and} \|\hat f(t)\|_{H^1} \leq C\langle t\rangle^{\delta}\eps.
\]
Here $\delta=\mathcal{O}(\eps^2)$ is a small parameter.  Observe that by Lemma~\ref{L:PW}, the assumptions \eqref{bootstrap} also guarantee that
\[
\|u(t)\|_{L^\infty} \lesssim \langle t\rangle^{-\frac12}\eps.
\]

Noting that $\|\hat f(t)\|_{L_\xi^\infty}\equiv \|w(t)\|_{L_\xi^\infty}$, we begin by using the expansion \eqref{dtw21}--\eqref{dtw23} to estimate $\partial_t w$ in in $L_\xi^\infty$. In particular, we will prove that if \eqref{bootstrap} holds, then 
\begin{equation}\label{dtw-bound}
\|\partial_t w\|_{L_\xi^\infty} \lesssim \langle t\rangle^{-1-\frac{1}{10}}\eps^3 \qtq{uniformly for}t\geq 1. 
\end{equation}

First, by \eqref{bootstrap} we immediately see that
\[
\tfrac{1}{2t(2t+1)}\| |\hat f|^2 \hat f\|_{L_\xi^\infty} \lesssim \langle t\rangle^{-2}\eps^3,
\]
which is acceptable. 

Next, using \eqref{bootstrap}, Hausdorff--Young, and Lemma~\ref{L:PW}, we estimate
\begin{align*}
\|\F e^{-it\Delta}(a|u|^2 u)\|_{L_\xi^\infty} & \lesssim \|a |u|^2 u\|_{L^1}  \lesssim \|a\|_{L^1}\|u\|_{L^\infty}^3  \lesssim_a \langle t\rangle^{-\frac32}\eps^3,
\end{align*}
which is acceptable. 

Finally, we turn to \eqref{dtw23}.  We begin by using the pointwise estimate
\[
|e^{ix}-1|\leq |x|^{\frac15} 
\]
to obtain
\begin{equation}\label{dtw23-temp}
\|\eqref{dtw23}\|_{L_\xi^\infty} \lesssim |t|^{-1-\frac15}\biggl\|\iint |\eta|^{\frac15}|\sigma|^{\frac15}|\F_2^{-1}\{G_\xi[f,f,f]\}(\eta,\sigma)|\,d\sigma\,d\eta\biggr\|_{L_\xi^\infty}.
\end{equation}
To estimate the right-hand side of \eqref{dtw23-temp}, we rely on the following general trilinear estimate.  We state the result in more generality than is needed here, as this formulation will be useful in the next section.

\begin{lemma}[Trilinear Estimate]\label{L:trilinear} Define $G_\xi(\cdot,\cdot,\cdot)$ as in \eqref{def:G}. Then
\[
\iint |\eta|^{\frac15}|\sigma|^{\frac15} |\F_2^{-1}\{G_\xi[f,g,h]\}(\eta,\sigma)|\,d\sigma\,d\eta \lesssim \|f\|_{H^{0,1}}\| g\|_{H^{0,1}}\| h\|_{H^{0,1}}
\]
uniformly in $\xi$.
\end{lemma}

\begin{proof} Recall that
\[
G_\xi[f,g,h](x,y) = \hat f(\xi-x)\hat{\bar g}(x-\xi+y)\hat h(\xi-y).
\]
Thus, writing $\int e^{iab}\,db = \delta_{a=0}$, we have
\begin{align}
\F_2^{-1}&\{G_\xi[f,g,h]\}(\eta,\sigma) \nonumber\\
& = \idotsint e^{i[x\eta+y\sigma-v(\xi-x)-z(x-\xi+y)-r(\xi-y)]}f(v)\bar g(z)h(r)\,dx\,dy\,dr\,dv\,dz\nonumber \\
& = \iiint \bar g(z)e^{iz\xi} f(v)e^{-iv\xi}e^{i[x(v+\eta-z)]}\biggl[\int h(r)e^{-ir\xi}\int e^{i[y(r+\sigma-z)]}\,dy\,dr\biggr]\,dx\,dv\,dz\nonumber \\
& = \int \bar g(z) h(z-\sigma) e^{i\xi\sigma}\biggl[\int f(v)e^{-iv\xi}\int e^{i[x(v+\eta-z)]}\,dx\,dv\biggr]\,dz\nonumber \\
& = \int f(z-\eta)\bar g(z) h(z-\sigma)  e^{i\xi[\eta+\sigma-z]}\,dz. \label{F2-id}
\end{align}
It follows that
\[
|\F_2^{-1}\{G_\xi[f,g,h]\}(\eta,\sigma)| \leq \int |f(z-\eta)h(z-\sigma)g(z)|\,dz 
\]
uniformly in $\xi$, and hence
\begin{align*}
\iint& |\eta|^{\frac15}|\sigma|^{\frac15}|\F_2^{-1}G_\xi[f,g,h]\}(\eta,\sigma)|\,d\sigma\,d\eta \\
& \lesssim \iiint |\eta|^{\frac15}|\sigma|^{\frac15} |f(z-\eta)h(z-\sigma)g(z)|\,dz\,d\sigma\,d\eta \\
& \lesssim \iiint [|z-\eta|^{\frac15}+|z|^{\frac15}][|z-\sigma|^{\frac15}+|z|^{\frac15}]|f(z-\eta)h(z-\sigma)g(z)|\,dz\,d\sigma\,d\eta
\end{align*}
uniformly in $\xi$. The result now follows from the fact that for any $0<c<\tfrac12$,
\[
\| |x|^{c} f\|_{L^1} \lesssim \|\langle x\rangle f\|_{L^2},
\]
which is a consequence of Cauchy--Schwarz.
\end{proof}

Continuing from \eqref{dtw23-temp} and applying Lemma~\ref{L:trilinear} and \eqref{bootstrap}, we obtain
\begin{align*}
\|\eqref{dtw23}\|_{L_\xi^\infty} \lesssim |t|^{-1-\frac15}\|f(t)\|_{H^{0,1}}^3 \lesssim |t|^{-1-\frac15+3\delta}\eps^3,
\end{align*}
which is acceptable (provided $\delta$ is sufficiently small).  This completes the proof of \eqref{dtw-bound}, which suffices to close the bootstrap estimate for $\hat f$ in $L^\infty$.

To complete the proof of \eqref{bootstrap}, it suffices to close the bootstrap estimate for $H^1$-norm of $\hat f$.  Without loss of generality, we estimate the $\dot H^1$-norm only.  

Using the Duhamel formula, we first write
\begin{align}
\partial_\xi \hat f(t,\xi) &  = \partial_\xi \hat u_0(\xi) \label{dxi-data} \\
& \quad - i\int_0^t \partial_\xi\bigl[\F e^{-is\Delta}\bigl(|u|^2 u\bigr)(\xi)\bigr]\,ds \label{dxi-homogeneous}\\
& \quad -i \int_0^t \partial_\xi\bigl[ \F e^{-is\Delta}\bigl(a|u|^2 u\bigr)(\xi)\bigr]\,ds. \label{dxi-inhomogeneous}
\end{align}
The term in \eqref{dxi-data} is $\mathcal{O}(\eps)$ in $L_\xi^2$, which is acceptable.

Using the same computations as above, we may write
\begin{equation}\label{dxi-homogeneous2}
\eqref{dxi-homogeneous} = -i\int_0^t\iint e^{2is\eta\sigma}\partial_\xi G_\xi[f(s),f(s),f(s)](\eta,\sigma)\,d\sigma\,d\eta\,ds.
\end{equation}
Recalling the definition of $G_\xi$ (see \eqref{def:G}), it follows from the product rule that $\partial_\xi G_\xi[f,f,f]$ is a linear combination of terms of the form $G_\xi[xf,f,f]$.  After distributing the derivative, we can use the identity
\[
x f(s) = xe^{-is\Delta} u(s) = e^{-is\Delta} J(s)u(s)
\]
(cf. \eqref{Galilean}) and undo the computations that led to \eqref{dxi-homogeneous2} to see that \eqref{dxi-homogeneous} may be written as a sum of terms of the form
\[
\int_0^t \F[e^{-is\Delta} \mathcal{O}(u^2) Ju](\xi)\,ds. 
\]
In particular, by \eqref{t01} and \eqref{bootstrap}, we may estimate
\begin{align*}
\|\eqref{dxi-homogeneous}\|_{L_\xi^2} & \lesssim \int_0^t \|u(s)\|_{L_\xi^\infty}^2 \|J(s)u(s)\|_{L^2}\,ds  \lesssim \int_0^t \langle s\rangle^{-1+\delta}\eps^3\,ds \lesssim \langle t\rangle^\delta \eps^3,
\end{align*}
which is acceptable.

% Recalling the definition of $G_\xi$ (see \eqref{def:G}), 
% The cubic term without inhomogeneity may be bounded in the standard manner: We begin by writing 
% \begin{align*}
% \F& e^{-it\Delta}\bigl(|u|^2 u\bigr)(\xi) \\
% & = \iint e^{it[\xi^2-(\xi-\eta)^2+(\eta-\sigma)^2-\sigma^2)]}\hat f(t,\xi-\eta)\hat{\bar{f}}(t,\eta-\sigma)\hat f(t,\sigma)\,d\sigma\,d\eta \\
% & = \iint e^{2it\eta\sigma}G_\xi[f(t),f(t),f(t)](\eta,\sigma)\,d\sigma\,d\eta.
% \end{align*}
% Taking $\partial_{\xi}$ and measuring resulting terms in the $L^2_\xi$ norm, one has \cjm{I may put in just a few more lines of estimates here...}
% \begin{align*}
% \bigl\|\partial_\xi\F& e^{-it\Delta}\bigl(|u|^2 u\bigr)(\xi) \bigr\|_{L^2_\xi}\\
% & =   \biggl\|\partial_\xi \iint e^{2it\eta\sigma}G_\xi[f(t),f(t),f(t)](\eta,\sigma)\,d\sigma\,d\eta   \biggr\|\\
% &\lesssim\|u\|^2\|xf\|^2\lesssim \epsilon^3 \frac{1}{s+1} s^\delta
% \end{align*}
% which together with Minkowski's inequality gives
% \begin{equation}
%     \biggl\|\partial_\xi \int_0^t \F e^{-it\Delta}\bigl(|u|^2 u\bigr)(\xi)\,ds \biggr\|_{L^2_\xi}\lesssim \int_0^t \epsilon^3 \frac{1}{s+1} s^\delta\,ds\lesssim \epsilon^3 t^\delta.
% \end{equation}

It remains to estimate \eqref{dxi-inhomogeneous}.  We begin by writing
\begin{align}
\partial_\xi \int_0^t  \F e^{-is\Delta}\bigl(a|u|^2 u\bigr)(\xi) \,ds&=\partial_\xi \int_0^t e^{is\xi^2} \F \bigl(a|u|^2 u\bigr)(\xi)\,ds \nonumber \\
&=\int_0^t e^{is\xi^2} \partial_\xi\F \bigl(a|u|^2 u\bigr)(\xi)\,ds \label{dxi-inhomogeneous2}\\
&\quad +2i\int_0^t \xi s e^{is\xi^2} \F \bigl(a|u|^2 u\bigr)(\xi)\,ds.\label{dxi-inhomogeneous3}
\end{align}
Using \eqref{t01} and \eqref{bootstrap}, we first estimate
\begin{align*}
\|\eqref{dxi-inhomogeneous2}\|_{L_\xi^2} \lesssim \int_0^t \|x\,a|u|^2 u\|_{L_x^2}\,ds \lesssim \int_0^t\|xa\|_{L^2} \|u\|_{L^\infty}^3\,ds \lesssim \int_0^t \eps^3 \langle s\rangle^{-\frac32}\,ds \lesssim\eps^3,
\end{align*}
which is acceptable. 

% Directly, one can bound
% \begin{equation}
%     \biggl\|\int_0^t e^{is\xi^2} \partial_\xi\F \bigl(a|u|^2 u\bigr)(\xi)\,ds\biggr\|_{L^2_\xi}\lesssim \epsilon^3\int_0^t \left(\frac{1}{\sqrt{1+s^2}}\right)^{3/2}\,ds\lesssim \epsilon^3.
% \end{equation}

Next, we let $\varphi$ be a smooth cutoff to $|\xi|\leq 1$ and decompose
\begin{align}
\eqref{dxi-inhomogeneous3}&=2i\int_0^t \xi \varphi(\xi) s e^{is\xi^2} \F \bigl(a|u|^2 u\bigr)(\xi)\,ds\label{dxi-inhomogeneous4}\\
&\quad +2i\int_0^t  [1-\varphi(\xi)]s e^{is\xi^2} \xi \F \bigl(a|u|^2 u\bigr)(\xi)\,ds.\label{dxi-inhomogeneous5}
\end{align}

Applying Lemma~\ref{lem:smoothing} (with $\phi(\xi)=\xi\varphi(\xi)$ and $\hat{F}(s,\xi)=s \F(a|u|^2u)$), \eqref{t01}, \eqref{bootstrap}, and Minkowski's integral inequality, we deduce that
\begin{align*}
\|\eqref{dxi-inhomogeneous4}\|_{L_\xi^2}&\lesssim \bigl\| s a|u|^2 u\bigr\|_{L^1_x L^2_s(\R\times[0,t])}\\
& \lesssim \|a\|_{L^1} \|s|u|^2 u\|_{L_x^\infty L_s^2(\R\times[0,t])} \\
& \lesssim \| s|u|^2 u\|_{L_s^2 L_x^\infty([0,t]\times\R)} \\
& \lesssim \eps^3\| s\langle s\rangle^{-\frac32}\|_{L_s^2([0,t])} \lesssim \eps^3\langle\log\langle t\rangle\rangle, 
\end{align*}
which is acceptable.  

Similarly, applying Lemma~\ref{lem:smoothing} (with $\phi(\xi)=1-\varphi(\xi)$ and $\hat F(s,\xi)=s\xi \F(a|u|^2 u)$), we find that
\begin{align}
\|\eqref{dxi-inhomogeneous5}\|_{L_\xi^2} &\lesssim \bigl\| s |u|^2 u\, \partial_x a\bigr\|_{L^1_x L^2_s(\R\times[0,t])}+\bigl\| s a|u|^2 \partial_x u\bigr\|_{L^1_x L^2_s(\R\times[0,t])}.
%\nonumber\\&\lesssim \eps^3\|\partial_x a\|_{L^1}\langle\log \langle t\rangle \rangle+\epsilon^3 \|\jx a\|_{L^2}\nonumber.  
\end{align}
For the first term, we proceed as we did for \eqref{dxi-inhomogeneous4}.  This yields
\[
\|s|u|^2 u\partial_x a\|_{L_x^1 L_s^2(\R\times[0,t])} \lesssim \eps^3\|\partial_x a\|_{L^1} \langle\log\langle t\rangle\rangle,
\]
which is acceptable. 

For the second term, we write
\[
s\partial_x u = \tfrac{1}{2i}[J(s)u(s)-xu(s)]
\]
Then, using \eqref{bootstrap} (noting that $\|Ju\|_{L^2} = \|\hat f\|_{\dot H^1}$ by \eqref{Galilean}), we estimate
\begin{align*}
\|s&a|u|^2 \partial_x u\|_{L_x^1 L_s^2(\R\times[0,t])} \\
& \lesssim \|\langle x\rangle a\|_{L^2} \||u|^2 \langle x\rangle^{-1}[Ju-xu]\|_{L_{s,x}^2([0,t]\times\R)} \\
& \lesssim \|u\|_{L_s^4 L_x^\infty([0,t]\times\R)}^2 \bigl\{ \|Ju\|_{L_s^\infty L_x^2([0,t]\times\R)}+\|\tfrac{x}{\langle x\rangle} u\|_{L_s^\infty L_x^2([0,t]\times\R)}\bigr\} \\
& \lesssim \eps^3\langle t\rangle^\delta,
\end{align*}
which is acceptable.
% \cgc{Maybewe we don't need to use Strichartz here (we didn't recall Strichartz)

% With bootstrap assumptions, we can directly bound 
% \begin{align}
%  \||u|^2 \langle x\rangle^{-1}[Ju-xu]\|_{L_{s,x}^2([0,t]\times\R)}&
%  \lesssim \|\|u\|_{L^\infty_x}^2 
%  \|\langle x\rangle^{-1}[Ju-xu]\|_{L^2_{x}}\|_{L_{s}^2([0,t])}\\ 
%   \lesssim \epsilon^3 \left\Vert\frac{1}{\langle t \rangle} \langle t\rangle^\delta\right\Vert_{L_{s}^2([0,t])}\lesssim \epsilon^3.
% \end{align}
% }

Combining the estimates above, we can close the bootstrap for the $H^1$-component of $\hat f$.  Thus the desired bounds for $\hat f$ hold for all $t\geq 0$, and in particular we obtain the bound \eqref{dtw-bound}.

With \eqref{dtw-bound} in hand, we obtain the establish the existence of $w_+$ in $L_\xi^\infty$ such that
\begin{equation}\label{quant-converge}
\|w(t)-w_+\|_{L_\xi^\infty} \lesssim \langle t\rangle^{-\frac{1}{10}}\eps^3
\end{equation}
uniformly for $t\geq 0$, which suffices to complete the proof of Theorem~\ref{T1}.

%%%%%%%%%%%%%%%%%%
\section{The Inverse Problem}\label{S:inverse}

The goal of this section is to prove Theorem~\ref{T2}.  Our first step is a careful analysis of the scattering map $u_0\mapsto S_a(u_0)$ for a fixed admissible inhomogeneity $a$.  

\begin{proposition}[Structure of $S_a$]\label{P:Sa} Let $a$ be admissible in the sense of Definition~\ref{D:admissible}. Let $\varphi\in\mathcal{S}(\R)$ and $\eps>0$ be sufficiently small.  Let $u:[0,\infty)\times\R\to\C$ be the solution to \eqref{nls} with $u|_{t=0}=\eps\varphi$. Then
\begin{equation}\label{Sa-structure}
\begin{aligned}
\langle S_a(\eps\varphi),\hat\varphi\rangle & = \eps\langle \hat{\varphi},\hat\varphi\rangle + \tfrac1{2i}\log(1+\tfrac{1}{2\eps})\langle |S_a(\eps\varphi)|^2S_a(\eps\varphi),\hat\varphi\rangle + \eps^3\mathcal{Q}_\eps[\varphi] \\
& \quad -i \eps^3\int_0^\infty\int_\R a(x)|e^{it\Delta}\varphi(x)|^4\,dx\,dt + \mathcal{O}(\eps^4),
\end{aligned}
\end{equation}
where
\begin{equation}\label{Q_eps}
\mathcal{Q}_\eps[\varphi]:=\int_\eps^\infty \tfrac{1}{2it}\iiint[e^{-i\frac{\eta\sigma}{2t}}-1]\varphi(z-\eta)\varphi(z-\sigma)\bar\varphi(z)\bar\varphi(z-\eta-\sigma)\,dz\,d\eta\,d\sigma\,dt. 
\end{equation}
\end{proposition}

\begin{proof} We write $u_0=\eps\varphi$ and let $u$ be the solution to \eqref{nls} with $u|_{t=0}=u_0$.  We define the profile $f(t)=e^{-it\Delta}u(t)$ and the modified profile $w(t) = e^{iB(t)}\hat f(t)$ as in \eqref{def:w}.  In particular, there exists $w_+\in L_\xi^\infty$ such that $w(t)\to w_+=S_a(u_0)$ in $L_\xi^\infty$ as $t\to\infty$. By construction, we have
\[
\|w_+\|_{L_\xi^\infty} \lesssim \eps.
\]

We begin by using \eqref{dtw21}--\eqref{dtw23} from the preceding section to write
\begin{align}
iw_+(\xi) & = i\hat u_0(\xi) + \int_0^\eps i\partial_t w(t,\xi)\,dt + \int_\eps^\infty \tfrac{1}{2t(2t+1)}|w(t,\xi)|^2 w(t,\xi) \,dt \label{w-expand1} \\
&\quad  + \int_\eps^\infty e^{iB(t,\xi)}\mathcal{G}_t[f(t),f(t),f(t)](\xi)\,dt \label{w-expand2} \\
& \quad + \int_\eps^\infty e^{iB(t,\xi)} \F[e^{-it\Delta}\{a|u(t)|^2 u(t)\}](\xi)\,dt,\label{w-expand3}
\end{align}
where
\begin{equation}\label{def:calG}
\mathcal{G}_t[f,g,h](\xi):= \tfrac{1}{2t}\iint [e^{-i\frac{\eta\sigma}{2t}}-1]\F_2^{-1}\{G_\xi[f,g,h]\}(\eta,\sigma)\,d\eta\,d\sigma,
\end{equation}
with $G_\xi(\cdot,\cdot,\cdot)$ as in \eqref{def:G}.

The term $\hat u_0(\xi)$ is $\mathcal{O}(\eps)$.   The analysis now proceeds by separating the remaining components in \eqref{w-expand1}--\eqref{w-expand2} that are $\mathcal{O}(\eps^3)$ in $L_\xi^\infty$ from those that are $o(\eps^3)$ as $\eps\to 0$. 

We first observe that by \eqref{dtwt01}, we have that
\[
\biggl\|\int_0^\eps \partial_t w\,dt\biggr\|_{L_\xi^\infty} \lesssim \eps^4.
\]

For the remaining term in \eqref{w-expand1}, we claim that
\begin{equation}\label{w-expand13}
\int_\eps^\infty \tfrac{1}{2t(2t+1)}|w(t,\xi)|^2w(t,\xi)\,dt = \tfrac12\log(1+\tfrac{1}{2\eps})|w_+(\xi)|^2 w_+(\xi) + \mathcal{O}(\eps^4)
\end{equation}
in $L_\xi^\infty$.  To see this, we use  \eqref{quant-converge} to estimate
\begin{align*}
\||w(t)|^2 w(t)- |w_+|^2 w_+ \|_{L_\xi^\infty} & \lesssim \{\|w(t)\|_{L_\xi^\infty}^2 + \|w_+\|_{L_\xi^\infty}^2\} \|w(t)-w_+\|_{L_\xi^\infty} \\
& \lesssim \eps^5\langle t\rangle^{-\frac1{10}},
\end{align*}
which yields
\begin{align*}
\biggl\| \int_\eps^\infty& \tfrac{1}{2t(2t+1)}\bigl[|w(t)|^2 w(t) - |w_+|^2 w_+\bigr]\,dt\biggr\|_{L_\xi^\infty} \\
& \lesssim \eps^5\int_\eps^\infty \tfrac{1}{2t(2t+1)}\langle t\rangle^{-\frac{1}{10}}\,dt \lesssim \eps^5|\log \eps| = \mathcal{O}(\eps^4). 
\end{align*}
As $\int_\eps^\infty \frac{1}{2t(2t+1)}\,dt = \tfrac12\log(1+\frac{1}{2\eps})$, we conclude that \eqref{w-expand13} holds.

Collecting the estimates so far, we have found
\begin{equation}\label{w-expand1-new}
\eqref{w-expand1} = i\hat u_0(\xi) + \tfrac12\log(1+\tfrac{1}{2\eps})|w_+(\xi)|^2 w_+(\xi)+\mathcal{O}(\eps^4). 
\end{equation}

We turn to the terms in \eqref{w-expand2}--\eqref{w-expand3}.  We first show that the phase $\exp\{iB(t)\}$ can be removed up to errors that are higher order in $\eps$ (at the price of logarithmic time growth).  In particular, we have
\begin{equation}\label{L:phase}
\|e^{iB(t)}-1\|_{L_\xi^\infty} \lesssim \|B(t)\|_{L_\xi^\infty} \lesssim \int_0^t \|\hat f(s)\|_{L_\xi^\infty}^2 \tfrac{ds}{2s+1} \lesssim \eps^2\langle\log \langle t\rangle\rangle.
\end{equation}

We now use \eqref{L:phase} to show that
\begin{equation}\label{w-expand2-temp}
\begin{aligned}
\eqref{w-expand2}+\eqref{w-expand3}& = \int_\eps^\infty \mathcal{G}_t[f(t),f(t),f(t)](\xi)\,dt\\
&\quad + \int_\eps^\infty\F[e^{-it\Delta}\{a|u(t)|^2 u(t)\}](\xi)\,dt + \mathcal{O}(\eps^4)
\end{aligned}
\end{equation}
uniformly in $\xi$.  To this end, we will verify the following two estimates:
\begin{align}
&\int_\eps^\infty \langle\log \langle t\rangle\rangle\|\mathcal{G}_t[f(t),f(t),f(t)]\|_{L_\xi^\infty}\,dt \lesssim \eps^{\frac{14}{5}}, \label{we2t1} \\
&\int_\eps^\infty \langle\log\langle t\rangle\rangle\|\F[e^{-it\Delta}\{a|u(t)|^2 u(t)\}]\|_{L_\xi^\infty}\,dt\lesssim \eps^3. \label{we2t2}
\end{align}

Using Lemma~\ref{L:trilinear}, we first have
\begin{align*}
|\eqref{we2t1}| &\lesssim \int_\eps^\infty\iint |t|^{-1-\frac15}\langle\log \langle t\rangle\rangle|\eta|^{\frac15}|\sigma|^{\frac15} |\F_2^{-1}\{G_\xi[f(t),f(t),f(t)]\}(\eta,\sigma)|\,d\eta\,d\sigma  \\
& \lesssim \int_\eps^\infty |t|^{-1-\frac15}\langle\log \langle t\rangle\rangle \|f(t)\|_{H^{0,1}}^3\,dt  \\
&\lesssim \eps^3 \int_\eps^\infty |t|^{-1-\frac15}\langle t\rangle^{3\delta}\langle \log \langle t\rangle\rangle\,dt \lesssim \eps^{\frac{14}{5}}.
\end{align*}
Next, by Hausdorff--Young and Lemma~\ref{L:PW},
\begin{align*}
|\eqref{we2t2}|&\lesssim \int_\eps^\infty \langle \log \langle t\rangle \rangle \|a |u(t)|^2u(t)\|_{L^1}\,dt \\
& \lesssim \int_\eps^\infty \langle\log \langle t\rangle\rangle \|a\|_{L^1} \|u(t)\|_{L^\infty}^3 \,dt  \lesssim \eps^3\int_\eps^\infty \langle\log \langle t\rangle \rangle\langle t\rangle^{-\frac32}\,dt \lesssim \eps^3. 
\end{align*}
Combining the preceding estimates with \eqref{L:phase}, we derive \eqref{w-expand2-temp}.

We now analyze each term in \eqref{w-expand2-temp} more closely.  We show that up to acceptable errors, we may replace the full solution with its initial data:

\begin{lemma}\label{L:cubic-main} The following approximations hold. First,
\begin{equation}
\int_\eps^\infty \mathcal{G}_t[f(t),f(t),f(t)] \,dt = \int_\eps^\infty \mathcal{G}_t[u_0,u_0,u_0]\,dt + \mathcal{O}(\eps^4)\label{G-replacement} \\
\end{equation}
in $L_\xi^\infty$. Next, for any test function $\psi$, 
\begin{equation}\label{A-replacement}\int_\eps^\infty \langle \F[e^{-it\Delta}\{a|u|^2 u\}],\psi\rangle\,dt =\int_\eps^\infty \langle a|e^{it\Delta}u_0|^2e^{it\Delta}u_0,e^{it\Delta}\check\psi\rangle\,dt + \mathcal{O}(\eps^4).
\end{equation}
\end{lemma}

\begin{proof} We begin with \eqref{G-replacement}.  Writing
\[
f(t)=u_0 + \int_0^t \partial_s f(s)\,ds, 
\]
we find that it suffices to prove that
\[
\int_\eps^\infty \mathcal{G}_t\biggl[g,h,\int_0^t \partial_s f(s)\,ds\biggr] \,dt = \mathcal{O}(\eps^4)
\]
in $L_\xi^\infty$, where
\[
g,h\in\biggl\{u_0,\int_0^t \partial_s f\,ds\biggr\}. 
\]
For each such term, we use Lemma~\ref{L:trilinear} to estimate 
\begin{align*}
\int_\eps^\infty& \biggl|\mathcal{G}_t\biggl[g,h,\int_0^t\partial_s f(s)\,ds\biggr](\xi)\biggr|\,dt \\
& \lesssim \int_\eps^\infty |t|^{-1-\frac15}|\eta|^{\frac15}|\sigma|^{\frac15}\biggl|\F_2^{-1}\biggl\{G_\xi\biggl[g,h,\int_0^t \partial_s f(s)\,ds\biggr]\biggr\}(\eta,\sigma)\biggr|\,d\eta\,d\sigma\,dt \\
& \lesssim \int_\eps^\infty |t|^{-1-\frac15} \|g\|_{H^{0,1}}\| h\|_{H^{0,1}} \biggl\| \int_0^t \partial_s f(s)\,ds\biggr\|_{H^{0,1}} \,dt 
\end{align*}
uniformly in $\xi$. Noting that the estimates in the preceding section imply
\[
\| \langle x\rangle \int_0^t \partial_s f(s)\,ds\|_{L^2} \lesssim \langle t\rangle^{3\delta}\eps^3,
\]
we see that
\[
\|g\|_{H^{0,1}}+\|h\|_{H^{0,1}}\lesssim \eps + \eps^3 \langle t\rangle^{3\delta}.
\]
It follows that
\begin{align*}
\int_\eps^\infty \biggl\|\mathcal{G}_t\biggl[g,h,\int_0^t\partial_s f(s)\,ds\biggr]\biggr\|_{L_\xi^\infty}\,dt  \lesssim \int_\eps^\infty |t|^{-1-\frac15}\{\eps^5\langle t\rangle^{3\delta}+\eps^9\langle t\rangle^{9\delta}\}\,dt  \lesssim \eps^{\frac{24}{5}},
\end{align*}
which is acceptable. 

We turn to \eqref{A-replacement}.  Fixing a test function $\psi$, we see that it suffices to prove 
\[
\int_\eps^\infty \langle a[|u|^2u - |e^{it\Delta}u_0|^2 e^{it\Delta}u_0],e^{it\Delta}\check{\psi}\rangle\,dt = \mathcal{O}(\eps^4).
\]
To prove this we first note that by the Duhamel formula for \eqref{nls}, we have 
\[
u(t)-e^{it\Delta}u_0 = N(t):=-i\int_0^t e^{i(t-s)\Delta}[(1+a)|u|^2 u](s)\,ds. 
\]
Thus by the dispersive estimate, Sobolev embedding, unitarity of $e^{it\Delta}$, and Lemma~\ref{L:PW}, we have 
\begin{align*}
\int_\eps^\infty& \bigl|\langle a[|u(t)|^2u(t) - |e^{it\Delta}u_0|^2 e^{it\Delta}u_0],e^{it\Delta}\check{\psi}\rangle\bigr|\,dt \\
& \lesssim \int_0^\infty\bigl\| a\bigl\{|u(t)|^2+|e^{it\Delta}u_0|^2\bigr\}|u(t)-e^{it\Delta}u_0|\cdot e^{it\Delta}\check{\psi}\bigr\|_{L_x^1}\,dt \\
& \lesssim \int_0^\infty \|a\|_{L_x^2}\|N(t)\|_{L_x^2} \bigl\{\|u(t)\|_{L_x^\infty}^2+\|e^{it\Delta}u_0\|_{L_x^\infty}^2\bigr\}\|e^{it\Delta}\check\psi\|_{L_x^\infty}\,dt \\
& \lesssim_a \eps^2 \int_0^\infty\langle t\rangle^{-\frac32}\|\psi\|_{H^{1,1}} \int_0^t \|1+a\|_{L_x^\infty}\| |u(s)|^2 u(s)\|_{L_x^2}\,ds\,dt \\
& \lesssim_{a,\psi} \eps^2\int_0^\infty \langle t\rangle^{-\frac32}\int_0^t \|u(s)\|_{L^\infty}^2 \|u(s)\|_{L^2}\,ds\,dt \\
& \lesssim_{a,\psi} \eps^5\int_0^\infty \langle t\rangle^{-\frac32}\int_0^t \langle s\rangle^{-1}\,ds\,dt \\
& \lesssim_{a,\psi}\eps^5\int_0^\infty\langle t\rangle^{-\frac32}\langle\log\langle t\rangle\rangle\,dt \lesssim_{a,\varphi}\eps^5,
\end{align*}
which is acceptable. \end{proof}

We return to the expansion for $w_+$ given in \eqref{w-expand1}--\eqref{w-expand3} and pair the expression with $\hat\varphi$. We insert \eqref{w-expand1-new} for \eqref{w-expand1} and combine \eqref{w-expand2-temp} with Lemma~\ref{L:cubic-main} to replace the terms \eqref{w-expand2}--\eqref{w-expand3}.  Recalling $u_0=\eps\varphi$, this yields
\begin{align*}
\langle S_a(\eps\varphi),\hat\varphi\rangle &  
= \eps\langle\hat\varphi,\hat\varphi\rangle + \tfrac{1}{2i}\log(1+\tfrac{1}{2\eps})\langle |w_+|^2 w_+,\hat\varphi\rangle \\
& \quad -i \eps^3\int_\eps^\infty \langle \mathcal{G}_t[\varphi,\varphi,\varphi],\hat\varphi\rangle\,dt \\
& \quad -i \eps^3\int_\eps^\infty a(x)|e^{it\Delta}\varphi(x)|^4\,dx\,dt + \mathcal{O}(\eps^4). 
\end{align*}
Comparing the identity above with \eqref{Sa-structure}, we see that to complete the proof of Proposition~\ref{P:Sa} it suffices to verify the following:
\begin{align}
&\int_0^\eps \int_\R a(x)|e^{it\Delta}\varphi(x)|^4\,dx\,dt = \mathcal{O}(\eps), \label{obvious} \\
& \int_\eps^\infty\tfrac{1}{i}\langle \mathcal{G}_t[\varphi,\varphi,\varphi],\hat\varphi\rangle\,dt = \mathcal{Q}_\eps[\varphi], \label{not-obvious}
\end{align}
where $\mathcal{Q}_\eps$ is as in \eqref{Q_eps}. 

The estimate \eqref{obvious} follows from the straightforward bound
\begin{align*}
\int_0^\eps\int a(x)|e^{it\Delta}\varphi(x)|^4\,dx\,dt & \lesssim \eps \|a\|_{L^\infty}\|e^{it\Delta}\varphi\|_{L_t^\infty L_x^4}^4 \lesssim_a \eps \|\varphi\|_{H^1}^4,
\end{align*}
where we have applied Sobolev embedding and unitary of $e^{it\Delta}$. 

The identity \eqref{not-obvious} follows from a straightforward calculation: recalling the definition in \eqref{def:calG} and the identity in \eqref{F2-id}, we have
\begin{align*}
\int_\eps^\infty& \tfrac{1}{i}\langle\mathcal{G}_t[\varphi,\varphi,\varphi],\hat\varphi\rangle\,dt \\
& = \int_\eps^\infty\tfrac{1}{2it}\iiint[e^{-i\frac{\eta\sigma}{2t}}-1]\F_2^{-1}\{G_\xi[\varphi,\varphi,\varphi]\}(\eta,\sigma)\bar{\hat\varphi}(\xi)\,d\xi\,d\eta\,d\sigma\,dt \\
& = \int_\eps^\infty \tfrac{1}{2it}\iiiint [e^{-i\frac{\eta\sigma}{2t}}-1]\varphi(z-\eta)\bar \varphi(z) \varphi(z-\sigma)\bar{\hat{\varphi}}(\xi)e^{i\xi[\eta+\sigma-z]}\,dz\,d\xi\,d\eta\,d\sigma\,dt \\
& = \int_\eps^\infty \tfrac{1}{2it}\iiiint [e^{-i\frac{\eta\sigma}{2t}}-1]\varphi(z-\eta)\bar\varphi(z)\varphi(z-\sigma)\bar\varphi(z-\eta-\sigma)\,dz\,d\eta\,d\sigma\,dt \\
& = \mathcal{Q}_\eps[\varphi],
\end{align*}
as desired. \end{proof}

We now turn to the proof of our main result, Theorem~\ref{T2}.

\begin{proof}[Proof of Theorem~\ref{T2}] We let $a$ and $b$ be admissible in the sense of Definition~\ref{D:admissible} and suppose that the modified scattering maps $S_a$ and $S_b$ agree on their common domain. We now fix $\varphi\in\mathcal{S}$ and sufficiently small $\eps>0$ and apply the main identity \eqref{Sa-structure} in Proposition~\ref{P:Sa} to both $S_a(\eps\varphi)$ and $S_b(\eps\varphi)$.  As $S_a(\eps\varphi)=S_b(\eps\varphi)$, this implies 
\[
\int_0^\infty\int_\R a(x)|e^{it\Delta}\varphi(x)|^4\,dx\,dt = \int_0^\infty \int_\R b(x)|e^{it\Delta}\varphi(x)|^4\,dx\,dt+\mathcal{O}(\eps)
\]
for any $\eps>0$. It follows that 
\[
\int_0^\infty \int_\R a(x)|e^{it\Delta}\varphi(x)|^4\,dx\,dt = \int_0^\infty \int_\R b(x)|e^{it\Delta}\varphi(x)|^4\,dx\,dt\qtq{for all}\varphi\in\mathcal{S}.
\]
Thus the proof of Theorem~\ref{T2} reduces to showing that if $a$ is admissible in the sense of Definition~\ref{D:admissible} and 
\begin{equation}\label{IF}
\int_0^\infty \int_\R a(x)|e^{it\Delta}\varphi(x)|^4\,dx\,dt = 0 \qtq{for all}\varphi\in\mathcal{S},
\end{equation}
then $a\equiv 0$. 

Given $\varphi\in\mathcal{S}$, we define the function
\[
K_\varphi(x) = \int_0^\infty |e^{it\Delta}\varphi(x)|^4\,dt 
\]
and first prove that $K_\varphi\in L^2$. To see this, we use Minkowski's integral inequality followed by the dispersive estimate and Sobolev embedding to estimate 
\begin{align*}
\bigl\|\,\|e^{it\Delta}\varphi\|_{L_t^4}^4 \bigr\|_{L_x^2} & \lesssim \|e^{it\Delta}\varphi\|_{L_x^8 L_t^4}^4  \\ 
& \lesssim \|e^{it\Delta}\varphi\|_{L_t^4 L_x^8}^4 \lesssim \|\langle t\rangle^{-\frac{3}{8}}\|_{L_t^4}^4 \|\varphi\|_{H^1}^4\lesssim_\varphi 1.\end{align*}

Now we specialize to the choice
\[
\varphi(x)=\exp\{-\tfrac{x^2}{4}\},\qtq{in which case} e^{it\Delta}\varphi(x) = \bigl[\tfrac{1}{1+it}\bigr]^{\frac12} \exp\{-\tfrac{x^2}{4(1+it)}\}
\]
(see \cite{Visan}).  In particular, 
\[
K_\varphi(x) = \int_0^\infty \tfrac{1}{1+t^2}\exp\bigl\{-\tfrac{x^2}{1+t^2}\bigr\}\,dt. 
\]

Now suppose that \eqref{IF} holds.  Then, by translation invariance for the linear Schr\"odinger equation, we have that
\[
\int_\R a(x)K_\varphi(x-x_0)\,dx = 0 \qtq{for all} x_0\in\R.
\]
Thus, to deduce that $a\equiv 0$, it suffices to verify that $\hat K_\varphi\neq 0$ almost everywhere.  In fact, for $\xi\neq 0$, we can compute $\hat K_\varphi(\xi)$ explicitly as a Gaussian integral:
\begin{align*}
\hat K_\varphi(\xi) & = \int_0^\infty(1+t^2)^{-1}\int_\R \exp\bigl\{-ix\xi -\tfrac{x^2}{1+t^2}\bigr\}\,dx \,dt \\
& = \sqrt{\pi}\int_0^\infty (1+t^2)^{-\frac12} \exp\bigl\{-\tfrac{\xi^2(1+t^2)}{4}\bigr\}\,dt.
\end{align*}
As $\hat K_\varphi(\xi)$ is the integral of a positive function, the result follows. \end{proof}

\end{document}